\documentclass[11pt]{article}
\usepackage[a4paper, tmargin=3cm, bmargin=3.0cm, lmargin=2.0cm, rmargin=2.0cm, textheight=24cm, textwidth=16cm]{geometry}
\usepackage{amsmath,fullpage,amsthm}
\usepackage{amscd}
\usepackage{amsthm,mathtools} \usepackage{amssymb}
\usepackage{latexsym}
\usepackage{eufrak}
\usepackage{euscript}
\usepackage{epsfig}
\usepackage{tikz}
\usepackage{graphics}
\usepackage{array}
\usepackage{enumerate} \usepackage{authblk}

\theoremstyle{theorem}
\newtheorem{theorem}{Theorem}[section]
\theoremstyle{corollary}
\newtheorem{corollary}{Corollary}[section]
\theoremstyle{lemma}
\newtheorem{lemma}{Lemma}[section]
\theoremstyle{definition}
\newtheorem{definition}{Definition}[section]
\theoremstyle{proof}
\theoremstyle{remark}
\newtheorem{remark}{Remark}[section]
\theoremstyle{example}
\newtheorem{example}{Example}[section]
\theoremstyle{observation}

\begin{document}
   \setcounter{Maxaffil}{2}
   \title{On the adjacency matrix of a complex unit gain graph}
   \author[a]{\rm Ranjit Mehatari\thanks{ranjitmehatari@gmail.com, mehatarir@nitrkl.ac.in}}
   \author[b]{\rm M. Rajesh Kannan\thanks{rajeshkannan1.m@gmail.com, rajeshkannan@maths.iitkgp.ac.in}}
\author[b]{\rm Aniruddha Samanta\thanks{aniruddha.sam@gmail.com}}
   \affil[a]{Department of Mathematics,}
   \affil[ ]{National Institute of Technology Rourkela,}
   \affil[ ]{Rourkela - 769008, India}
   \affil[ ]{ }
   \affil[b]{Department of Mathematics,}
   \affil[ ]{Indian Institute of Technology Kharagpur,}
   \affil[ ]{Kharagpur-721302, India.}
   \maketitle
\begin{abstract}
A complex unit gain  graph  is a simple graph in which each orientation of an edge is given a complex number with modulus $1$ and its inverse is assigned to the opposite orientation of the edge. In this article, first we establish bounds for the eigenvalues of the complex unit gain graphs. Then we study some of the properties of the adjacency matrix of complex unit gain graph in connection with the characteristic and the permanental polynomials. Then we establish spectral properties of the adjacency matrices of  complex unit gain graphs. In particular, using Perron-Frobenius theory,  we establish a characterization for bipartite graphs in terms of the set of eigenvalues of gain graph and the set of eigenvalues of the underlying  graph. Also, we derive an equivalent condition on the gain so that the eigenvalues of the gain graph and the eigenvalues of the underlying graph are the same.
\end{abstract}

{\bf AMS Subject Classification(2010):} 05C50, 05C22.

\textbf{Keywords:} Gain graph, Characteristic polynomial,  Perron-Frobenius theory,  Bipartite graph, Balanced gain graph.

\section{Introduction}
Let $G=(V,E)$ be a simple, undirected, finite graph with the vertex set $V(G)=\{v_1,v_2,\ldots,v_n\}$ and the edge set $E(G) \subseteq V \times V$.  If two vertices $v_i$ and $v_j$ are adjacent, we write $v_i\sim v_j$,   and the edge between them is denoted by $e_{ij}$. The degree of the vertex $v_i$ is denoted by $d_i$. The $(0,1)$-\textit{adjacency matrix}  or simply the \textit{adjacency matrix} of $G$ is an $n \times n$ matrix, denoted by $A(G)=[a_{ij}]$, whose rows and columns are indexed by the vertex set of the graph and the entries are defined by
$$a_{ij}=\begin{cases}
1, &\text{if }v_i\sim v_j,\\
0, &\text{otherwise.}\end{cases}$$

The adjacency matrix of a graph is one of the well studied matrix class  in the field of spectral graph theory. For more details about the study of classes of matrices associated with graphs, we refer to \cite{Bap-book, Brou, Chung, Cve2, Cve1, Mer}.

The notion of gain graph was introduced in \cite{Zas1}. For a given graph $G$ and a group $\mathfrak{G}$, first orient the edges of the graph $G$. For each oriented edge $e_{ij}$ assign a value (the \textit{gain} of the edge $e_{ij}$) $g$ from $\mathfrak{G}$ and assign $g^{-1 }$ to the orientated edge $e_{ji}$. If the group is taken to be the multiplicative group of unit complex numbers, the graph is called the \textit{complex unit gain graph}. Now let us recall the definition of complex unit gain graphs \cite{Reff1}. The set of all oriented edges of the graph $G$ is denoted by $\overrightarrow{E}(G)$.

\begin{definition}
    A $\mathbb{T}$-gain graph (or complex unit gain graph) is a triple $\Phi=(G,\mathbb{T},\varphi)$, where
    \begin{itemize}
        \item[(i)] $G=(V,E)$ is a simple finite graph,
        \item[(ii)] $\mathbb{T}$ is the unit complex circle, i.e., $\mathbb{T}=\{z\in\mathbb{C}:|z|=1\}$, and
        \item[(iii)] the map $\varphi:\overrightarrow{E}(G)\rightarrow\mathbb{T}$  is such that $\varphi(e_{ij})=\varphi(e_{ji})^{-1}$.
    \end{itemize}
    Since, we consider $\mathbb{T}$-gain graphs throughout this paper,  we use $\Phi=(G,\varphi)$ instead of $\Phi=(G,\mathbb{T},\varphi)$.
\end{definition}
 The study of the spectral properties of $\mathbb{T}$-gain graphs is interesting because this generalizes the theory of adjacency matrix for undirected graphs. In \cite{Reff1},  the author introduced the adjacency matrix $A(\Phi)=[a_{ij}]_{n\times n}$ for a $\mathbb{T}$-gain graph $\Phi$ and provided some important spectral properties of $A(\Phi)$. The entries of $A(\Phi)$ are given by
$$a_{ij}=\begin{cases}
\varphi(e_{ij}),&\text{if } \mbox{$v_i\sim v_j$},\\
0,&\text{otherwise.}\end{cases}$$
If $v_i$ is adjacent to $v_j$, then $a_{ij} = \varphi(e_{ij}) = \varphi(e_{ji})^{-1}
= \overline{\varphi(e_{ji})} = \overline{a_{ji}}$ . Thus the matrix $A(\Phi)$ is Hermitian, and its eigenvalues are real. Let $\sigma(A(\Phi))$ denote the set of eigenvalues of the matrix $A(\Phi)$.

Particular cases of the notion of adjacency matrix of  $\mathbb{T}$-gain graphs were considered  with different weights  in the literature \cite{Bap,Kat}. In \cite{Kat}, the authors considered complex weighted graphs with the weights taken from $\{\pm1,\pm i\}$, and characterized unicyclic graph having strong reciprocal eigenvalue property. In \cite{Ger}, the authors studied some of the properties of the characteristics polynomial for gain graphs. For some interesting spectral properties of gain graphs, we refer to \cite{Reff1, Reff2, Ger,  Zas2}. When $\varphi(e_{ij})=1$ for all $e_{ij}$, then $A(\Phi)=A(G)$. Thus we can consider $G$ as a $\mathbb{T}$-gain graph and we write this by $(G,1)$. Recently the notion of incident matrix and Laplacian matrix for $\mathbb{T}$-gain graphs have been studied \cite{Reff1, lap-gain}. From the above discussion it is clear that, the spectral theory of $\mathbb{T}$-gain graphs generalizes the spectral theory of undirected  graphs and some weighted graphs.


Next we recall some of the  definitions and notation which we needed.  For more details we refer to \cite{Reff1,Reff2,Zas4,Zas1,Zas3,Zas2}.
\begin{definition}
    The \textit{gain} of a cycle (with some orientation) $C={v_1v_2\ldots v_lv_1}$, denoted by $\varphi(C)$,  is defined as the product of the gains of its edges, that is
    $$\varphi(C)=\varphi(e_{12})\varphi(e_{23})\cdots\varphi(e_{(l-1)l})\varphi(e_{l1}).$$
    A cycle $C$ is said to be \textit{neutral} if $\varphi(C)=1$, and a gain graph is said to be \textit{balanced} if all its cycles are neutral. For a cycle $C$ of $G$, we denote the real part of the gain of $C$ by $\mathfrak{R}(C)$, and it is independent of the orientation.
\end{definition}
\begin{definition}
    A function from the vertex set of $G$ to the  complex unit circle $\mathbb{T}$ is called a \textit{switching function}. We say that, two gain graphs $\Phi_1=(G,\varphi_1)$ and $\Phi_2=(G,\varphi_2)$ are \textit{switching equivalent}, written as $\Phi_1\sim\Phi_2$, if there is a switching function $\zeta:V\rightarrow\mathbb{T}$ such that $$\varphi_2(e_{ij})=\zeta(v_i)^{-1}\varphi_1(e_{ij})\zeta(v_j).$$

    The switching equivalence of two gain graphs can be defined in the following equivalent way:
    Two gain graphs $\Phi_1=(G,\varphi_1)$ and $\Phi_2=(G,\varphi_2)$ are switching equivalent, if there exists a diagonal matrix $D_\zeta$ with diagonal entries from $\mathbb{T}$, such that
    $$A(\Phi_2)=D_\zeta^{-1}A(\Phi_1)D_\zeta.$$
\end{definition}

\begin{definition}
    A \textit{potential function} for $\varphi$ is a function $\psi:V\rightarrow\mathbb{T}$, such that for each edge $e_{ij}$, $\varphi(e_{ij})=\psi(v_i)^{-1}\psi(v_j).$
\end{definition}

\begin{theorem}\cite{Zas1}\label{Zas1}
    Let $\Phi=(G,\varphi)$ be a $\mathbb{T}$-gain graph. Then the following statements are equivalent:
    \begin{itemize}
        \item[(i)] $\Phi$ is balanced,
        \item[(ii)] $\Phi\sim(G,1)$,
        \item[(iii)] $\varphi$ has a potential function.
    \end{itemize}
\end{theorem}

The following necessary condition for switching equivalence is known.
\begin{theorem}\cite{Reff1}
    \label{Th1}
    Let $\Phi_1=(G,\varphi_1)$ and $\Phi_2=(G,\varphi_2)$ be $\mathbb{T}$-gain graphs. If $\Phi_1\sim\Phi_2$, then $\sigma(A(\Phi_1))=\sigma(A(\Phi_2))$.
\end{theorem}

In Section \ref{spec-bip}, we construct a counter example to show that the above necessary condition is not sufficient.

The next theorem gives a necessary condition for the for two $\mathbb{T}$-gain graphs to be switching equivalent.
\begin{theorem}\cite{Reff2}
    \label{Th2}
    Let $\Phi_1=(G,\varphi_1)$ and $\Phi_2=(G,\varphi_2)$ be two $\mathbb{T}$-gain graphs. If $\Phi_1\sim\Phi_2$ , then  for any cycle $C$ in $G$, $\varphi_1(C)=\varphi_2(C)$ holds.
\end{theorem}

\begin{definition}
    The \emph{characteristic polynomial} of a gain graph, denoted by $P_{\Phi}(x)$, is defined as $P_{\Phi}(x) = \det (x I - A(\Phi))= x^n+a_1x^{n-1}+\cdots+a_n$. The \emph{permanental polynomial} of a gain graph, denoted by $Q_{\Phi}(x)$, is defined as $Q_{\Phi}(x) =\text{per}(x I-A(\Phi))= x^n+b_1x^{n-1}+\cdots+b_n$. The characteristic and the permanental polynomials of the underlying graph $G$ is denoted by $P_{G}(x)$ and $Q_{G}(x)$, respectively. Some important and interesting properties of $P_G(x)$ and $Q_G(x)$ can be found in \cite{Cve1,Mer1}.
\end{definition}
\begin{definition}
    A \textit{matching} in a graph $G$ is a set of edges such that no two of them  have a vertex in common. {The number of edges in a matching is called the \emph{size of that matching}. The collection of all matchings of size $k$ in a graph $G$ is denoted by $\mathcal{M}_k(G)$. We define $m_k(G)=|\mathcal{M}_k(G)|$ with the convention that $m_0(G)=1$.}  A matching is called \textit{maximal} if it is not contained in any other matching. The largest possible cardinality among all matchings is called the \textit{matching number} of $G$.\end{definition}

Let $\mathbb{C}^{n \times n}$ denote the set of all $n \times n$ matrices with complex entries. For a matrix $A = (a_{ij}) \in \mathbb{C}^{n \times n}$, define $|A| = (|a_{ij}|)$. Let $\rho(A)$ denote the spectral radius of the matrix $A$. The following results about nonnegative matrices will be useful in Section \ref{spec-bip}.
\begin{theorem}\cite[Theorem 8.1.18]{Horn}\label{Th0.3a}
Let $A, B\in \mathbb{C}^{n \times n}$ and  suppose that $B$ is nonnegative. If $|A| \leq B$, then $\rho(A) \leq \rho(|A|) \leq \rho(B).$
\end{theorem}
\begin{theorem}\cite[Theorem 8.4.5]{Horn}\label{Th0.3}
Let $A, B\in \mathbb{C}^{n \times n}$. Suppose $A$ is nonnegative and irreducible, and $A\geq |B|$. Let $\lambda=e^{i \theta} \rho(B)$ be a given maximum-modulus eigenvalue of $B$. If $\rho(A)=\rho(B)$, then there is a diagonal unitary  matrix $D \in \mathbb{C}^{n \times n}$ such that $B=e^{i\theta}DAD^{-1}$.\\
\end{theorem}

This article is organized as follows: In Section \ref{bounds}, we provide some results on the eigenvalue bounds for the adjacency matrix of $\mathbb{T}$-gain graphs. In Section \ref{coeff-char-per}, we study some of the properties of  the coefficients of characteristic and permanental polynomials of gain adjacency matrices. In Section \ref{spec-bip}, we focus on the study of the spectral properties of  $\mathbb{T}$-gain graphs. We establish an equivalent condition for the equality of set of eigenvalues of a $\mathbb{T}$-gain graph  and that of the underlying graph.  Finally, we give a characterization for the bipartite graphs in terms of the eigenvalues of the gains and the eigenvalues of the underlying graph.\\

\section{Eigenvalue bounds for $\mathbb{T}$-gain graphs}\label{bounds}
For any complex square matrix $B$, we use $\lambda(B)$  to denote its eigenvalues (or simply $\lambda$ when there is only one matrix under consideration). Let $B$ be a complex square matrix of order $n$ with real eigenvalues. We arrange the eigenvalues of $B$ as
$$\lambda_n\leq\lambda_{n-1}\leq\ldots\leq\lambda_2\leq\lambda_1.$$
We now recall some important results associated to complex square matrices having real eigenvalues.
\begin{theorem}\cite[Theorem 2.1]{Wolk}
    \label{bound_thm1}
    Let $B$ be an $n\times n$ complex matrix with real eigenvalues,
    and let $$r=\frac{\text{trace } B}{n}\ \ \text{ and }\ \ s^2=\frac{\text{trace } B^2}{n}-r^2, $$
    then
    $$r-s(n-1)^\frac{1}{2}\leq\lambda_n\leq r-s/(n-1)^\frac{1}{2},~\mbox{and } $$
    $$r+s/(n-1)^\frac{1}{2}\leq\lambda_1\leq r+s(n-1)^\frac{1}{2}.$$
\end{theorem}
\begin{theorem}\cite{Horn}
    \label{bound_thm2}
    Let $B$ be a Hermitian matrix of order n and let $B_r$ be any principal submatrix of $B$ of order $r$. Then for $1\leq k\leq r$,
    $$\lambda_{n+k-r}(B)\leq\lambda_k(B_r)\leq\lambda_k(B).$$
\end{theorem}
It is well known that (\cite{Bap-book,Cve1}),  the $(i,j)$-th entry of the $k$-th power adjacency matrix of a simple graph provide the number of $k$-walks (walks of length $k$) from the vertex $i$ to $j$. The next lemma is the counter part of the above statement for $\mathbb{T}$-gain graphs.
\begin{lemma}
    \label{bound_lem1}
    Let $\Phi$ be a $\mathbb{T}$-gain graph. Then the $(i,j)$-th entry $a_{ij}^{(k)}$ of $A(\Phi)^k$ is the sum of gains of all $k$-walks from the vertex $i$ to the vertex $j$.
\end{lemma}
\begin{proof}
Let $A(\Phi)^k=(a_{ij}^{(k)})$. Then $$a_{ij}^{(k)}=\sum a_{ii_1}a_{i_1i_2}\ldots a_{i_{k-2}i_{k-1}}a_{i_{k-1}j},$$ where $1\leq i_1,i_2,\ldots, i_{k-1}\leq n$. Now, any term in the right side is nonzero if and only if $v_i\sim v_{i_1}$, $v_{i_1}\sim v_{i_2}$, $\ldots$, $v_{i_{k-1}}\sim v_j$. Hence the result follows.
\end{proof}
Next, we derive lower and upper bounds for the largest and smallest eigenvalues of  $\mathbb{T}$-gain graph in terms of the number of edges and vertices of the underlying graph. The following result is known for the adjacency matrices \cite{Bap-book}-- however, the proof technique the different from the adjacency matrix case.
\begin{theorem}
    \label{key}
    Let $\Phi$ be a $\mathbb{T}$-gain graph with the underlying graph $G$. If $G$ has $n$ vertices and $m$ edges, then the  smallest and largest eigenvalues of $A(\Phi)$ satisfy
    $$-\sqrt{\frac{2m(n-1)}{n}}\leq\lambda_n\leq-\sqrt{\frac{2m}{n(n-1)}},$$
    and
    $$\sqrt{\frac{2m}{n(n-1)}}\leq\lambda_1\leq\sqrt{\frac{2m(n-1)}{n}}.$$Both of the above bounds are tight.
\end{theorem}
\begin{proof}
   Since the $i$-th diagonal entry of $A(\Phi)^2$ is $\sum_{v_i\sim v_k}\varphi(e_{ik})\varphi(e_{ki})=d_i$, we have
    $r=\frac{\text{trace }A(\Phi)}{n}=0\ \ \text{and}\ \ s^2=\frac{\text{trace } A(\Phi)^2}{n}-r^2=\frac{2m}{n}.$
    Thus, by Theorem \ref{bound_thm1}, we have
    $$-\sqrt{\frac{2m(n-1)}{n}}\leq\lambda_n\leq-\sqrt{\frac{2m}{n(n-1)}},$$
    and
    $$\sqrt{\frac{2m}{n(n-1)}}\leq\lambda_1\leq\sqrt{\frac{2m(n-1)}{n}}.$$
    Hence the proof is completed.
    
    Let $ \Phi=(K_{n}, \varphi) $ be a $ \mathbb{T} $-gain graph with $ \varphi(e)=-1 $ for all edges, then both of the left equality occurs and if $ \varphi(e)=1 $ for all edges then both right equality attains.
\end{proof}

In the next theorem, we derive a lower bound for the largest eigenvalue of the $\mathbb{T}$-gain graph.
\begin{theorem}
    \label{key}
    Let $\Phi$ be a $\mathbb{T}$-gain graph with the underlying graph $G$. Then
    $$\lambda_1\geq \sqrt[3]{\frac{6}{n}\sum_{C\in\mathcal{C}_3(G)}\mathfrak{R}(C)},$$
    where $\mathcal{C}_3(G)$ denotes the collection of all cycles of length 3 in G.
\end{theorem}
\begin{proof}
    If $\lambda_1$ is the largest eigenvalue of $A(\Phi)$, then $\lambda_1^3$ is the largest eigenvalue of $A(\Phi)^3$. So, we have  $\lambda_1^3\geq\frac{1}{n}\text{ trace }A(\Phi)^3.$
    By Lemma \ref{bound_lem1}, the $i$-th diagonal entry of $A(\Phi)^3$ is $ a_{ii}^{(3)}=\sum_{v_i\sim v_j\sim v_k\sim v_i}\varphi(e_{ij})\varphi(e_{jk})\varphi(e_{ki}) = 2\sum_{C\in \mathcal{C}_3(i)}\mathfrak{R}(C),$
     where $\mathcal{C}_3(i)$ denotes collection of all triangles which contains the vertex $i$. Now, since each triangle contains 3 vertices, we have
 $$     \lambda_1^3 \geq \frac{1}{n}\text{ trace }A(\Phi)^3 =\frac{1}{n}\sum_{i=1}^na_{ii}^{(3)}=\frac{6}{n}\sum_{C\in\mathcal{C}_3(G)}\mathfrak{R}(C).$$
  The result follows by taking cube root on both sides.
\end{proof}
Next result gives a lower bound for the spectral radius of the $\mathbb{T}$-gain graph in terms of the degrees of its vertices. This extends   \cite[Theorem 1]{ravi-kum} for the gain graphs.
\begin{theorem}
    Let $\Phi$ be a $\mathbb{T}$-gain graph with the underlying graph $G$. Then
    $$\sigma\geq\frac{1}{\sqrt{2}}\max_{i<j}\sqrt{d_i+d_j+\sqrt{(d_i-d_j)^2+4|a_{ij}^{(2)}|}},$$
    where $\sigma=\max|\lambda_i|$ and $|a_{ij}^{(2)}|$ is defined as in Lemma \ref{bound_lem1}.
\end{theorem}
\begin{proof}
    We have, $\sigma^2=\lambda_1(A(\Phi)^2)$. Let $A(\Phi)^2[i,j]=\left[\begin{array}{cc}
    a_{ii}^{(2)}&a_{ij}^{(2)}\\a_{ji}^{(2)}&a_{jj}^{(2)}
    \end{array}
    \right]$ be a principal submatrix of $A(\Phi)^2$. Then, by Theorem \ref{bound_thm2}, we have
    $\lambda_1(A(\Phi)^2)\geq\lambda_1(A(\Phi)^2[i,j]).$
    By Lemma \ref{bound_lem1}, $a_{ii}^{(2)}=d_i$, $a_{jj}^{(2)}=d_j$ and
    $\overline{a_{ji}^{(2)}}=a_{ij}^{(2)}=\sum_{v_i\sim v_k\sim v_j}\varphi(e_{ik})\varphi(e_{kj}).$
    Thus, \begin{eqnarray*}
        \lambda_1(A(\Phi)^2[i,j])&=&\frac{1}{2}\bigg{[}d_i+d_j+\sqrt{(d_i+d_j)^2-4(d_id_j-|a_{ij}^{(2)}|^2)}\bigg{]},\\
        &=&\frac{1}{2}\bigg{[}d_i+d_j+\sqrt{(d_i-d_j)^2+4|a_{ij}^{(2)}|^2}\bigg{]}.
    \end{eqnarray*}
    Since the above relation holds for all $i\neq j$, we have
    $$\sigma\geq\frac{1}{\sqrt{2}}\max_{i<j}\sqrt{d_i+d_j+\sqrt{(d_i-d_j)^2+4|a_{ij}^{(2)}|^2}}.$$
\end{proof}

\section{Characteristic and permanental polynomial of $\mathbb{T}$-gain graphs}\label{coeff-char-per}

In this section first we recall the some known definitions. Then, we compute the coefficients of the characteristic and permanental polynomials in terms of the gains of the edges.
\begin{definition}
Let $K_n$ denote the complete graph on $n$ vertices, and $K_{p,q}$ denote the complete bipartite graph on $p + q$ vertices with the vertex partition $V = V_1 \cup V_2 $, $|V_1| = p$ and $|V_2| = q$ . A graph $G$ is called an \textit{elementary graph}, if each of its component is either a $K_2$ or a cycle.
\end{definition}
Let $\mathcal{H}(G)$ denote the collection of all spanning elementary subgraphs of a graph $G$, and for any $H \in\mathcal{H}(G)$, let $\mathcal{C}(H)$ denote the collection of cycles in $H$. In \cite{Ger}, authors considered gain graphs with gains are taken from an arbitrary group. The following two results can be proved by taking the gains from the multiplicative group $\mathbb{T}$ in Corollary 2.3 and  Theorem 2.2 of \cite{Ger}, respectively.
\begin{theorem}
    \label{det1}
    Let $\Phi$ be a $\mathbb{T}$-gain graph with the underlying graph $G$. Then
    \begin{equation}
    \det A(\Phi)=\sum_{H\in\mathcal{H}(G)}(-1)^{n-p(H)}2^{c(H)}\prod_{C\in \mathcal{C}(H)}\Re(C),
    \end{equation}
    where $p(H)$ is the number of components in $H$ and  $c(H)$ is the number of cycles in $H$.
\end{theorem}
\begin{corollary}
    \label{cor1}Let $\Phi$ be any $\mathbb{T}$-gain graph with the underlying graph $G$. Let $P_\Phi(x)=x^n+a_1x^{n-1}+\cdots+a_n$ be the characteristics polynomial of $\Phi$. Then
    $$a_i=\sum_{H\in\mathcal{H}_i(G)}(-1)^{p(H)}2^{c(H)}\prod_{C\in \mathcal{C}(H)}\Re(C),$$
    where $\mathcal{H}_i(G)$ is the set of elementary subgraphs of $G$ with $i$ vertices.
\end{corollary}
\begin{proof}
    We have $$a_i=(-1)^i\sum i\times i \text{ principal minors}.$$
    Now, the result follows from Theorem \ref{det1}.
\end{proof}

In the next theorem, we compute coefficients of the permanental polynomial of the gain graphs. This result is an extension of the well known Harary's theorem in the context of permanents to the $\mathbb{T}$-gain graphs.  The proof is similar to that of \cite[Theorem 2.3.2]{Cve1} and \cite{har1}. For the sake of completeness we include a proof here.

\begin{theorem}
    \label{per1}
    Let $\Phi$ be a $\mathbb{T}$-gain graph with the underlying graph $G$. Then
    \begin{equation}
    \text{per } A(\Phi)=\sum_{H\in\mathcal{H}(G)}2^{c(H)}\prod_{C\in \mathcal{C}(H)}\Re(C),
    \end{equation}
    where $c(H)$ is the number of cycles in $H$.
\end{theorem}
\begin{proof}
    We have \begin{equation}\label{det-form}\text{per }A(\Phi)=\sum_{\sigma\in S_n}b_\sigma,\end{equation}
    where $b_\sigma=a_{1\sigma(1)}a_{2\sigma(2)}\ldots a_{n\sigma(n)}$, and $S_n$ denotes the collection of all permutations on the set $\{1,2,\ldots,n\}$. Since $a_{ii}=0$ for all $i=1,2,\ldots,n$, we have $b_{\sigma}\neq 0$ only if $v_i\sim v_{\sigma(i)}$ for all $i=1,2,\ldots,n$. Let $\gamma_1\gamma_2\ldots\gamma_r$ be the cycle decomposition of the permutation $\sigma$, where $\gamma_i$'s are disjoint cycles of length at least two. Thus, the decomposition of $\sigma$ determines an elementary spanning subgraph $H$ of $G$, whenever  $b_\sigma\neq0$. Now, let us calculate the value of $b_\sigma$. For this we consider each $\gamma_i$'s in the decomposition of $\sigma$.

    If $\gamma_i=(jk)$ is a transposition, then $a_{jk}$ and $a_{kj}$ occurs in the expression of $b_\sigma$. Also note that $a_{jk}a_{kj}=1$. If $\gamma_i=(i_1i_2\ldots i_k)$ is a $k$-cycle, then $b_\sigma$ contains $a_{i_1i_2}a_{i_2i_3}\ldots a_{i_ki_1}$. Let $C$ denote the cycle $v_{i_1}v_{i_2}\ldots v_{i_k}v_{i_1}$, then $a_{i_1i_2}a_{i_2i_3}\ldots a_{i_ki_1}=\varphi(C).$  By combining all these possibilities, we get $b_\sigma=\prod_{C\in\mathcal{C}(H)}\varphi(C).$ Let $\gamma_1,\gamma_2,\ldots,\gamma_s$ be the cycles of length at least $3$ in the decomposition of $\sigma$. Now if we replace any of $\gamma_1,\gamma_2,\ldots,\gamma_s$ by $\gamma_1^{-1},\gamma_2^{-1},\ldots,\gamma_s^{-1}$, then the  sign of the obtained permutation is same that of $\sigma$. Let $\sigma'$ be the any of the $2^{c(H)}$ permutations, namely $\gamma_1^{\pm 1}\gamma_2^{\pm 1}\ldots\gamma_s^{\pm1}\gamma_{s+1}\ldots\gamma_r$, which have the same sign that of  $\sigma$. Then the contribution of $b_{\sigma'}$ to the sum (\ref{det-form}) is $\prod_{C\in\mathcal{C}(H)}\varphi(C)^{\pm1}$.\\
    Therefore,
    \begin{eqnarray*}
        \text{per } A(\Phi)&&=\sum_{H\in\mathcal{H}(G)}\prod_{C\in \mathcal{C}(H)}[\varphi(C)+\varphi(C)^{-1}],\\
        &&=\sum_{H\in\mathcal{H}(G)}2^{c(H)}\prod_{C\in \mathcal{C}(H)}\Re(C).
    \end{eqnarray*}
\end{proof}
The following result is an extension of the well known Sach's coefficient theorem (in the context of permanents) to the $\mathbb{T}$-gain graphs.
\begin{corollary}
    \label{cor2}Let $\Phi$ be any $\mathbb{T}$-gain graph with the underlying graph $G$. Let $Q_\Phi(x)=x^n+b_1x^{n-1}+\cdots+b_n$ be the permanental polynomial of $\Phi$. Then
    $$b_i= (-1)^i\sum_{H\in\mathcal{H}_i(G)}2^{c(H)}\prod_{C\in \mathcal{C}(H)}\Re(C),$$
    where $\mathcal{H}_i(G)$ is the set of elementary subgraphs of $G$ with $i$ vertices.
\end{corollary}

Now, let us calculate the characteristic polynomial of certain $\mathbb{T}$-gain graph using the previous results. The following graph is a particular case of class of graphs known as windmill graphs \cite{gut-sci,ani-ranj-amc}.
\begin{example}[Star of triangles]{\rm
    Let $S_m^\Delta$ denote the star with $m$ triangles, that is, the end vertices of $m$ copies of $K_2$ are joined to single vertex (see Figure \ref{fig3}). We label the vertices of $S_m^\Delta$ such that for each $ l\in \{1, \dots, m\}$, $v_1v_{2l}v_{2l+1}v_{1}$ denote a triangle in it. Let $\Phi=(S_m^\Delta,\varphi)$ be a $\mathbb{T}$-gain graph with
    $\varphi(v_1v_{2l}v_{2l+1}v_{1})=e^{i\theta_l},$ for $1\leq l\leq m$. Let
    $\alpha=2[\cos\theta_1+\cos\theta_2+\cdots+\cos\theta_l].$
    Let $P_\Phi(x)=x^n+a_1x^{n-1}+\cdots+a_n$ be the characteristics polynomial of $\Phi$. {We have $a_1=0$, and we calculate $a_i$ for $1\leq i\leq 2m+1$. Since all the triangles in $S_m^\Delta$ shares the vertex $v_1$. Thus any elementary subgraph with even number (say $2l$) of vertices must be a matching of size $l$. A $l$-matching in $S_m^\Delta$ is either a set of $l$ edges of the form $v_{2i}v_{2i+1}$ or a set consisting of an edge of the form $v_1v_{2j}$ (or $v_1v_{2j+1}$) together with $l-1$ edges of the form $v_{2i}v_{2i+1}$, $j\neq i$.

        Thus, by using Corollary \ref{cor1}, we have
        \begin{eqnarray*}
            a_{2l}&=&(-1)^lm_l(S_m^\Delta)\\&=&(-1)^l\bigg[\binom{m}{l}+2m\binom{m-1}{l-1}\bigg],
        \end{eqnarray*}
    where $m_l(S_m^\Delta)$ denote the number of $l$ matchings of $S_m^\Delta$.

        On the other hand an elementary subgraph with odd number (say $2l+1$) of vertices must contain a triangle $v_1v_{2j}v_{2j+1}v_1$ and $l-1$ edges of the form $v_{2i}v_{2i+1}$, $j\neq i$. Therefore}
    $a_{2l+1}=(-1)^l\binom{m-1}{l-1}\alpha.$
    As a special case if $\alpha=0$, then $a_{2l+1}=0$ so in this case eigenvalues of $\Phi$ are symmetric about 0.

In a similar way, we can show that $b_{2l}=\binom{m}{l}+2m\binom{m-1}{l-1}
   $ and $  b_{2l+1}=-\binom{m-1}{l-1}\alpha.$}
\end{example}
\begin{figure}[h]
    \centering
    \includegraphics[height=4.5cm]{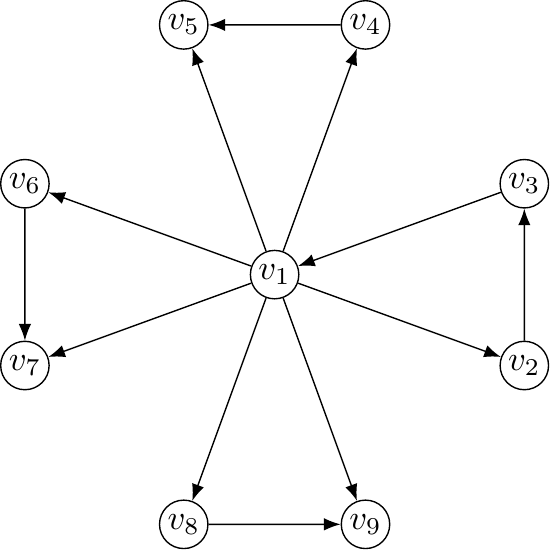}
  \caption{The graph $S^\Delta_4$.} 
   \label{fig3} 
\end{figure}

\begin{example}{\rm
    Let $G=K_4$ and $\varphi$ is taken in such a way that any three vertices $v_{i_1},v_{i_2},v_{i_3}$ with $i_1<i_2<i_3$ we have
    $$\varphi(v_{i_1}v_{i_2}v_{i_3}v_{i_1})=e^{i\theta}.$$
    Then for any cycle $C$ with three vertices we have $\mathfrak{R}(C)=\cos\theta$ and the gain of any cycle which can be written as a product of gains of two cycles of order three. The gains of the cycles of order four are
    \begin{eqnarray*}
        &&\varphi(v_{i_1}v_{i_2}v_{i_3}v_{i_4}v_{i_1})=\varphi(v_{i_1}v_{i_2}v_{i_3}v_{i_1})\varphi(v_{i_1}v_{i_3}v_{i_4}v_{i_1})=e^{2i\theta},\\
        &&\varphi(v_{i_1}v_{i_2}v_{i_4}v_{i_3}v_{i_1})=\varphi(v_{i_1}v_{i_2}v_{i_4}v_{i_1})\varphi(v_{i_1}v_{i_4}v_{i_3}v_{i_1})=1, ~\mbox{and}\\
        &&\varphi(v_{i_1}v_{i_3}v_{i_2}v_{i_4}v_{i_1})=\varphi(v_{i_1}v_{i_3}v_{i_2}v_{i_1})\varphi(v_{i_1}v_{i_2}v_{i_4}v_{i_1})=1.
    \end{eqnarray*}
    Therefore, the characteristic polynomial of $\Phi$ is
    $$P_\Phi(x)=x^4-6x^2-8x\cos\theta+(1-4\cos^2\theta).$$}
\end{example}

Next, we study some of the relationships between the characteristic and permanental polynomials of the adjacency matrix of the  $\mathbb{T}$-gain graph and that of the underlying graph $G$, respectively.

In the following theorem, we observe that, for a tree $G$, the characteristic (resp., the permanental) polynomial of the adjacency matrix and the characteristic polynomial (resp., the permanental) polynomial of any $\mathbb{T}$-gain graph $(G, \phi)$ are the same.
\begin{theorem}
    Let $\Phi=(G,\varphi)$ be a $\mathbb{T}$-gain graph. If G is a tree, then
    \begin{itemize}
        \item[(i)]
        $P_\Phi(x)=P_G(x)$, and
        \item[(ii)]
        $Q_\Phi(x)=Q_G(x)$.
    \end{itemize}
\end{theorem}
\begin{proof}
    The result follows from Corollary \ref{cor1} and Corollary \ref{cor2}.
\end{proof}

For a unicyclic graph $G$, we establish a relationship between the characteristic polynomial of the adjacency matrix of a graph $G$ and the characteristic polynomial of any $\mathbb{T}$-gain graph $(G, \phi)$ in terms of the length of the cycle and the matching number of the graph $G$. 
\begin{theorem}\label{match-gain}
    \label{unic}
    Let $G$ be an unicyclic graph with the cycle $C$ of length $m$. If $\Phi=(G,\varphi)$ be a $\mathbb{T}$-gain graph such that  $\varphi(C)=e^{i\theta}$. If $P_\Phi(x)$ and  $P_G(x)$ denote the characteristic polynomial of $\Phi$ and $G$, respectively, then
    $$P_\Phi(x)=P_G(x)+2(1-\cos \theta)\sum_{i=0}^k(-1)^i{m_i(G-C)}x^{n-m+2i},$$
    where $k$ is the matching number of $G-C.$
\end{theorem}
\begin{proof}
    Let $a_i(\Phi)$ and $a_i(G)$ denote the  coefficients of $x^{n-i}$ in $P_\Phi(x)$ and $P_G(x)$, respectively. We have $a_i(\Phi)=a_i(G),$ for all $i<m$, and $a_m(\Phi)=a_i(G)+2(1-\cos\theta).$

    Also note that, for $1\leq i\leq2k$,$$a_{m+i}(\Phi)=\begin{cases}a_{m+i}(G), &\text{if  }i \text{ is odd},\\
    a_{m+i}(G)+(-1)^{\frac{i}{2}}{m_\frac{i}{2}(G- C)}2(1-\cos\theta), &\text{if  }i \text{ is even}.
    \end{cases}$$
    Again $G-C$ has matching number $k$ implies $G$ has no elementary subgraph of order greater than $m+2k$ which contains a cycle. Thus, for all $i>m+2k$, we have
    $a_i(\Phi)=a_i(G).$
    Therefore,
    $$P_\Phi(x)=P_G(x)+2(1-\cos \theta)\sum_{i=0}^k(-1)^i{m_i(G-C)}x^{n-m+2i},$$
    which completes the proof.
\end{proof}
Next theorem is a counterpart of Theorem \ref{match-gain} for the permanental polynomial of a graph. 
\begin{theorem}
    Let $G$ be an unicyclic graph with the cycle $C$ of length $m$. If $\Phi=(G,\varphi)$ is a $\mathbb{T}$-gain graph such that  $\varphi(C)=e^{i\theta}$. If $Q_\Phi(x)$ and  $Q_G(x)$ denote the permanental  polynomial of $\Phi$ and $G$, respectively, then
    $$Q_\Phi(x)=Q_G(x)+(-1)^{m+1}2(1-\cos \theta)\sum_{i=0}^k{m_i(G-C)}x^{n-m+2i},$$
    where $k$ is the matching number of $G-C.$
\end{theorem}
\begin{proof}
    Similar to the proof of Theorem \ref{unic}.
\end{proof}
In the following result,  for a unicyclic graph $G$, we establish a relationship between the determinant (resp., permanent) of the adjacency matrix of a graph $G$ and the determinant (resp., permanent) of any $\mathbb{T}$-gain graph $(G, \phi)$.
\begin{corollary}
    Let $G$ be unicyclic and $\Phi$ be any gain graph with the underlying graph $G$. Then
    $$\det A(\Phi)=\begin{cases}
    \det A(G),&\text{if }2k\neq n-m,\\
    \det A(G)+(-1)^{k}2{m_k(G-C)}(1-\cos\theta),&\text{if }2k=n-m,
    \end{cases}$$
    and, similarly,
    $$\text{per} A(\Phi)=\begin{cases}
    \text{per} A(G),&\text{if }2k\neq n-m,\\
    \text{per} A(G)+(-1)^{m+1}2{m_k(G-C)}(1-\cos\theta),&\text{if }2k=n-m.
    \end{cases}$$
\end{corollary}

\section{Spectral properties of $\mathbb{T}$-gain graphs}\label{spec-bip}
In this section we study some of the spectral properties of the bipartite $\mathbb{T}$-gain graphs. First, let us establish that for a bipartite $\mathbb{T}$ gain graph $A(\Phi)$, the set of all eigenvalues $\sigma(A(\Phi))$ is symmetric about $0$. Proof of the unweighted case can be found in \cite{Bap-book}.
\begin{theorem}
    If $G$ is bipartite $\mathbb{T}$-gain graph, then the eigenvalues of $A(\Phi)$ are symmetric about $0$.
\end{theorem}
\begin{proof}
    Let $V=\{X,Y\}$ be the bipartition of the vertex set of $G$ such that $|X|=p.$  Let $\lambda$ be an eigenvalue of $A(\Phi)$ and $x=[x_1\ \cdots\ x_p\ x_{p+1}\ \cdots\ x_n]^T$ be a corresponding eigenvector. Then the vector $x'=[x_1\ \cdots\ x_p\ -x_{p+1}\ \cdots\ -x_n]^T$ is non-zero and  satisfies
    $A(\Phi)x'=-\lambda x'.$
    Therefore, $-\lambda$ is also an eigenvalue of $A(\Phi)$.
\end{proof}
\begin{remark}
    \label{rem1}{\rm
    The converse of the above theorem need not be true for gain graphs. Consider the complete graph $K_3$ on three vertices with edge weights are equal to $i$. Then $$A(\Phi)=\left[\begin{array}{ccc}
    0&i&i\\-i&0&i\\-i&-i&0
    \end{array}
    \right].$$
    The eigenvalues of $A(\Phi)$ are $0,\ \pm\sqrt{3}$. But the underlying graph is not bipartite.}
\end{remark}
\begin{remark}
    \label{rem2}{\rm
    It is known that $r$-regular graph has the eigenvalue $r$.  The example in  Remark \ref{rem1} also shows that,  $r$ may not be an eigenvalue for a $r$-regular $\mathbb{T}$-gain graph.}
\end{remark}
It is known that the eigenvalues of $A(G)$ are symmetric about $0$ if and only if $G$ is bipartite (see \cite{Cve1, Bap-book}). But Remark \ref{rem1} shows that this result need not true for $\mathbb{T}$-gain graphs. In the following theorem, we establish a sufficient condition for a $\mathbb{T}$-gain graph  with eigenvalues are  symmetric with respect to origin to be bipartite.
\begin{theorem}
    Let $\Phi=(G,\varphi)$ be a $\mathbb{T}$-gain graph such that the eigenvalues of $A(\Phi)$ are symmetric about $0$. If
    $$\sum_{C\in\mathcal{C}_i(G)}\mathfrak{R}(C)\neq 0,$$ where $\mathcal{C}_i(G)$ denotes the set of all cycles on $i$ vertices, then $G$ is bipartite.
\end{theorem}
\begin{proof}
    It is sufficient to prove that $G$ does not have any odd cycles.
    Let $P_\Phi(x)=x^n+a_1x^{n-1}+\cdots+a_n$ be the characteristics polynomial of $A(\Phi)$. Since, the eigenvalues of $A(\Phi)$ are symmetric about $0$, we have $a_i=0$ whenever $i$ is odd.  Now, $a_3=-2\sum_{C\in\mathcal{C}_3(G)}\mathfrak{R}(C)=0,$ and by the assumption $\sum_{C\in\mathcal{C}_3(G)}\mathfrak{R}(C)\neq0.$  Thus $G$ does not have any triangle. Since $G$ does not contain $K_3$, the cycles with $5$ vertices are the only elementary  subgraphs on $5$ vertices. A similar argument shows that $G$ does not have any cycles of length $5$. Proceeding in this way, we can prove $G$ does not contain any odd cycles.
\end{proof}
In the next theorem, we establish an upper bound for the largest eigenvalue of a bipartite graph.
\begin{theorem}
    Let $\Phi$ be a $\mathbb{T}$-gain graph with the underlying graph G. If $G=K_{p,q}$, then $\lambda_1(A(\Phi))\leq \sqrt{pq}$.  Equality holds if and only if $\Phi$ is balanced.
\end{theorem}
\begin{proof}
    We have, $\sum\lambda_i^2=\Big{[}\sum\lambda_i\Big{]}^2-2\sum_{i\neq j}\lambda_i\lambda_j=2pq.$
    Since, the eigenvalues of $A(\Phi)$ are symmetric about $0$, we get
    $2\lambda_1^2\leq 2pq,\ \text{ and hence }\ \lambda_1\leq \sqrt{pq}.$
    Now, let us prove the necessary and sufficient condition for the equality. If $\Phi$ is balanced, then $\sigma(A(\Phi))=\sigma(K_{p,q})$, and hence $\lambda_1(A(\Phi))=\sqrt{pq}.$
    Conversely, let $\lambda_1(A(\Phi))=\sqrt{pq}$. Suppose that $\Phi$ is not balanced. Then there exists a smallest number $l\geq 2$ such that $\Phi$ contains cycles $C_1, C_2,\ldots,C_k$ of length $2l$ with $\varphi(C_i)\neq 1.$
    Now computing coefficient of $x^{n-2l}$ in the characteristics polynomial of $\Phi$, we get
    $$ a_{2l}=2k-2\sum_{i=1}^k\mathfrak{R}(C_i) \neq 0.$$
    Thus $\lambda_2>0$, and hence $\lambda_1<\sqrt{pq}$. This contradicts the fact that $\lambda_1=\sqrt{pq}.$ Therefore $\Phi$ must be balanced.
\end{proof}
From Theorem \ref{Th1}, it is known that if two gain graphs are switching equivalent, then they have the same set of eigenvalues. The converse of this statement is not true in general, i.e., if the set of all eigenvalues of two $\mathbb{T}$-gain graphs {(with the same underlying graph $G$)} are the same, then they need not be switching equivalent. Consider $G$ as in the Figure \ref{Fig1}.
    \begin{figure}[h]
        \centering
        \includegraphics[height=4.5cm]{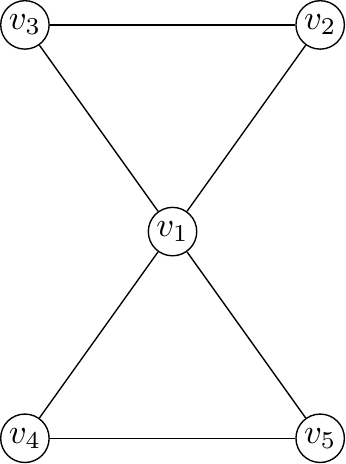}
        \caption{A graph with 5 vertices.}
        \label{Fig1}
    \end{figure}
    The graph G contains two cycles, namely, $C_1=\{v_1,v_2,v_3,v_1\}$ and
    $C_1=\{v_1,v_4,v_5,v_1\}$. We construct $\Phi_1$ and $\Phi_2$ so that
    $$\begin{array}{cc}
    \varphi_1(C_1)=i,&\ \varphi_1(C_2)=1\\ \varphi_2(C_1)=\frac{1+i\sqrt{3}}{2},&\ \varphi_2(C_2)=\varphi_2(C_1)=\frac{1+i\sqrt{3}}{2}.
    \end{array}$$
    Let $a_i(1)$ and $a_i(2)$ denote the coefficient of $x^{n-i}$ in the characteristics polynomial of $\Phi_1$ and $\Phi_2$, respectively. Then, by using Corollary \ref{cor1}, we have
 \begin{eqnarray*}
        &&a_1(1)=a_1(2)=0,\\
        &&a_2(1)=a_2(2)=-6,\\
        &&a_3(1)=a_3(2)=-2,\\
        &&a_4(1)=a_4(2)=1,\\
        &&a_5(1)=a_5(2)=2.
\end{eqnarray*}
Therefore $\sigma(A(\Phi_1))=\sigma(A(\Phi_2))$. But, by Theorem \ref{Th2}, $\Phi_1$ and $\Phi_2$ are not switching equivalent.

In the remaining part of this section, we study when the set of all eigenvalues  or the spectral radius of  $A(\Phi)$, for some $\Phi$, equals to the the set of all eigenvalues  or the spectral radius of the underlying graph, respectively.
%
%
\begin{lemma}
    Let $\Phi=(G,\varphi)$ be a $\mathbb{T}$-gain(connected) graph, then $ \rho(A(\Phi))\leq\rho(A(G)) $.
\end{lemma}
\begin{proof}
Since $ |A(\Phi)|\leq A(G) $, then, by  Theorem \ref{Th0.3a}, we have $ \rho(A(\Phi))\leq \rho(A(G))$.
\end{proof}

In the next theorem, we establish an equivalent condition for  $\rho(A(\Phi))=\rho(A(G))$.

\begin{theorem} \label{Th0.4}
   Let $\Phi=(G,\varphi)$ be a $\mathbb{T}$-gain(connected) graph, then $\rho(A(\Phi))=\rho(A(G))$ if and only if either $\Phi$ or $ -\Phi $ is balanced.
\end{theorem}
\begin{proof}
   If $\Phi$ or $ -\Phi $ is balanced, then $\rho(A(\Phi))=\rho(A(G))$. Conversely, suppose  that $\rho(A(\Phi))=\rho(A(G))$. Let $\lambda_n\leq\lambda_{n-1}\leq \dots \leq\lambda_1$ be the eigenvalues of $A(\Phi)$. Since $A(\Phi)$ is Hermitian,  either $\rho(A(\Phi))=\lambda_1$ or $\rho(A(\Phi))=-\lambda_n$.

  Now we have the following two cases:\\
   \textbf{Case 1:} Suppose that $\rho(A(\Phi))=\lambda_1$. Then, by Theorem \ref{Th0.3}, there is a diagonal unitary matrix $D\in  \mathbb{C}^{n \times n}$ such that  $A(\Phi)=DA(G)D^{-1}$. Hence $\Phi\sim (G,1)$.
   Therefore, by Theorem \ref{Zas1}, $\Phi$ is balanced. \\
   \textbf{Case 2:} If $\rho(A(\Phi))=-\lambda_n$, then $\lambda_n=e^{i \pi} \rho(A(\Phi))$. By Theorem \ref{Th0.3}, we have $A(\Phi)=e^{i \pi}DA(G)D^{-1}$, for some diagonal unitary matrix $D\in  \mathbb{C}^{n \times n}$. Thus $A(-\Phi)=  DA(G)  D^{-1}$. Hence, $(-\Phi)\sim (G,1)$. Thus, $-\Phi$ is balanced.
\end{proof}

\begin{theorem}\label{lm0.5}
     Let $\Phi=(G,\varphi)$ be a $\mathbb{T}$-gain(connected) graph. Then
    \begin{enumerate}
         \item[(i)] If $ G $ is bipartite, then whenever $ \Phi $ is balanced implies $-\Phi $ is balanced.
        \item[(ii)] If $ \Phi $ is balanced implies $ -\Phi $ is balanced for some gain $ \Phi $, then the graph is bipartite.
   \end{enumerate}
\end{theorem}

\begin{proof}
   (i) Suppose $ G $ is bipartite and $ \Phi $ is balanced. Then due to the absence of odd cycles,  $ -\Phi $ is balanced.

   (ii) Let $ \Phi $ be a balanced cycle such that $-\Phi $ is balanced. Suppose that $ G $ is not bipartite.
    Then, any odd cycle in $G$ can not be  balanced with respect to $ -\Phi $, which contradicts the assumption.  Thus $ G $ must be bipartite.
\end{proof}

In the next theorem, we answer the following problem: which gains  adjacency matrices are cospectral to the adjacency matrix of underlying graph.

\begin{theorem}\label{Th0.6}
       Let $\Phi=(G,\varphi)$ be a $\mathbb{T}$-gain(connected) graph. Then, $ \sigma(A(\Phi)) =\sigma(A(G))$ if and only if $ \Phi $ is balanced.
\end{theorem}
\begin{proof}
    If $ \sigma(A(\Phi)) =\sigma(A(G))$, then $ \rho (A(\Phi))=\rho(A(G))$. Now, by Theorem \ref{Th0.4}, we have either $ \Phi  $ or $ -\Phi $ is balanced. If $\Phi$ is balanced, then we are done.  Suppose that $ -\Phi $ is balanced, then $ -A(G) $ and $ A(\Phi) $ have the same set of eigenvalues. Hence $ \sigma(A(G)) = \sigma(-A(G))$. Thus, we have $ G $  is bipartite. Therefore, by Theorem \ref{lm0.5}, $ \Phi $ is balanced.
\end{proof}

In the next theorem, we derive a characterization for bipartite graphs in terms gains.
\begin{theorem}\label{Th0.7}
        Let $\Phi=(G,\varphi)$ be a $\mathbb{T}$-gain(connected) graph. Then, $ G $ is bipartite if and only if  $ \rho(A(\Phi))=\rho(A(G))$ implies $\sigma(A(\Phi))=\sigma(A(G)) $ for every gain $ \varphi $.
\end{theorem}
\begin{proof}
    Suppose $\rho(A(\Phi))=\rho(A(G))$ implies $\sigma(A(\Phi))=\sigma(A(G)) $  for any gain $ \varphi $. Let $ \Phi $ be balanced. We shall prove that $-\Phi$ is also balanced. By Theorem \ref{Th0.6}, we have $ \sigma(A(\Phi))=\sigma(A(G))$. Thus $ \rho(A(\Phi))=\rho(A(G))$. Also $ \rho(A(\Phi))=\rho(A(-\Phi)) $ implies $ \rho(A(-\Phi))=\rho(A(G)) $. Thus  $\sigma(A(-\Phi))=\sigma(A(G))$, and hence, by Theorem \ref{Th0.6}, $ -\Phi$ is balanced.  Now, by Theorem \ref{lm0.5}, $G$ is bipartite.

    Conversely, let $ G $ be a bipartite graph, and $\Phi$ be such that $ \rho(A(\Phi))=\rho(A(G)) $. Now,  by Theorem \ref{Th0.4} and Theorem \ref{lm0.5}, we  have $ \Phi $ to be balanced. Hence $
     \sigma(A(\Phi))=\sigma(A(G)) $.
\end{proof}


\section*{Acknowledgment}The authors are thankful to the referee for valuable comments and suggestions. 
Ranjit Mehatari is funded by NPDF (File no.- PDF/2017/001312), SERB, India. M. Rajesh
Kannan thanks the Department of Science and Technology, India, for financial support through
the Early Carrier Research Award (ECR/2017/000643) . Aniruddha Samanta thanks University
Grants Commission(UGC) for financial support in the form of a junior research fellowship.
\bibliographystyle{amsplain}
\bibliography{gain-ref}

\providecommand{\bysame}{\leavevmode\hbox to3em{\hrulefill}\thinspace}
\providecommand{\MR}{\relax\ifhmode\unskip\space\fi MR }
\providecommand{\MRhref}[2]{%
  \href{http://www.ams.org/mathscinet-getitem?mr=#1}{#2}
}
\providecommand{\href}[2]{#2}
\begin{thebibliography}{10}

\bibitem{Bap-book}
R.~B. Bapat, \emph{Graphs and matrices}, second ed., Universitext, Springer,
  London; Hindustan Book Agency, New Delhi, 2014. \MR{3289036}

\bibitem{Bap}
R.~B. Bapat, D.~Kalita, and S.~Pati, \emph{On weighted directed graphs}, Linear
  Algebra Appl. \textbf{436} (2012), no.~1, 99--111. \MR{2859913}

\bibitem{Brou}
Andries~E. Brouwer and Willem~H. Haemers, \emph{Spectra of graphs},
  Universitext, Springer, New York, 2012. \MR{2882891}

\bibitem{Chung}
Fan R.~K. Chung, \emph{Spectral graph theory}, CBMS Regional Conference Series
  in Mathematics, vol.~92, Published for the Conference Board of the
  Mathematical Sciences, Washington, DC; by the American Mathematical Society,
  Providence, RI, 1997. \MR{1421568}

\bibitem{Cve2}
Drago\v{s} Cvetkovi\'{c}, Peter Rowlinson, and Slobodan Simi\'{c},
  \emph{Signless {L}aplacians of finite graphs}, Linear Algebra Appl.
  \textbf{423} (2007), no.~1, 155--171. \MR{2312332}

\bibitem{Cve1}
\bysame, \emph{An introduction to the theory of graph spectra}, London
  Mathematical Society Student Texts, vol.~75, Cambridge University Press,
  Cambridge, 2010. \MR{2571608}

\bibitem{gut-sci}
Ivan Gutman and Irene Sciriha, \emph{Spectral properties of windmills}, Graph
  Theory Notes N. Y. \textbf{38} (2000), 20--24. \MR{1751021}

\bibitem{har1}
Frank Harary, \emph{The determinant of the adjacency matrix of a graph}, SIAM
  Rev. \textbf{4} (1962), 202--210. \MR{0144330}

\bibitem{Wolk}
George P.H.~Styan Henry~Wolkowicz, \emph{Bounds for eigenvalues using traces},
  Linear Algebra and its Applications \textbf{29} (1980), 471--506.

\bibitem{Horn}
Roger~A. Horn and Charles~R. Johnson, \emph{Matrix analysis}, Cambridge
  University Press, 2012.

\bibitem{Kat}
Debajit Kalita and Sukanta Pati, \emph{A reciprocal eigenvalue property for
  unicyclic weighted directed graphs with weights from {$\{\pm 1,\pm i\}$}},
  Linear Algebra Appl. \textbf{449} (2014), 417--434. \MR{3191876}

\bibitem{ravi-kum}
Ravinder Kumar, \emph{Bounds for eigenvalues of a graph}, J. Math. Inequal.
  \textbf{4} (2010), no.~3, 399--404. \MR{2724935}

\bibitem{ani-ranj-amc}
Ranjit Mehatari and Anirban Banerjee, \emph{Effect on normalized graph
  {L}aplacian spectrum by motif attachment and duplication}, Appl. Math.
  Comput. \textbf{261} (2015), 382--387. \MR{3345287}

\bibitem{Mer}
Russell Merris, \emph{Laplacian matrices of graphs: a survey}, Linear Algebra
  Appl. \textbf{197/198} (1994), 143--176, Second Conference of the
  International Linear Algebra Society (ILAS) (Lisbon, 1992). \MR{1275613}

\bibitem{Mer1}
Russell Merris, Kenneth~R. Rebman, and William Watkins, \emph{Permanental
  polynomials of graphs}, Linear Algebra Appl. \textbf{38} (1981), 273--288.
  \MR{636042}

\bibitem{Reff1}
Nathan Reff, \emph{Spectral properties of complex unit gain graphs}, Linear
  Algebra Appl. \textbf{436} (2012), no.~9, 3165--3176. \MR{2900705}

\bibitem{Reff2}
\bysame, \emph{Oriented gain graphs, line graphs and eigenvalues}, Linear
  Algebra Appl. \textbf{506} (2016), 316--328. \MR{3530682}

\bibitem{Ger}
K.~A.~Germina Shahul~Hameed, \emph{Balance in gain graphs {\textendash} a
  spectral analysis}, Linear Algebra and its Applications \textbf{436} (2012),
  no.~5, 1114--1121.

\bibitem{lap-gain}
Yi~Wang, Shi-Cai Gong, and Yi-Zheng Fan, \emph{On the determinant of the
  {L}aplacian matrix of a complex unit gain graph}, Discrete Math. \textbf{341}
  (2018), no.~1, 81--86. \MR{3713385}

\bibitem{Zas4}
Thomas Zaslavsky, \emph{Signed graphs}, Discrete Appl. Math. \textbf{4} (1982),
  no.~1, 47--74. \MR{676405}

\bibitem{Zas1}
\bysame, \emph{Biased graphs. {I}. {B}ias, balance, and gains}, J. Combin.
  Theory Ser. B \textbf{47} (1989), no.~1, 32--52. \MR{1007712}

\bibitem{Zas3}
\bysame, \emph{Biased graphs. {IV}. {G}eometrical realizations}, J. Combin.
  Theory Ser. B \textbf{89} (2003), no.~2, 231--297. \MR{2017726}

\bibitem{Zas2}
\bysame, \emph{Matrices in the theory of signed simple graphs}, Advances in
  discrete mathematics and applications: {M}ysore, 2008, Ramanujan Math. Soc.
  Lect. Notes Ser., vol.~13, Ramanujan Math. Soc., Mysore, 2010, pp.~207--229.
  \MR{2766941}

\end{thebibliography}
\end{document}